\documentclass[11pt, reqno]{amsart}
\usepackage{amsfonts,amssymb,amscd,amstext,mathrsfs,color}
\usepackage[dvipsnames]{xcolor}
\usepackage{graphicx}
\usepackage{tikz}
\usepackage{hyperref}
\usepackage{mathtools}

\numberwithin{equation}{section}

\newtheorem{theorem}{Theorem}[section]
\newtheorem{corollary}[theorem]{Corollary}
\newtheorem{lemma}[theorem]{Lemma}
\newtheorem{proposition}[theorem]{Proposition}

\newtheorem{conjecture}[theorem]{Conjecture}

\theoremstyle{definition}
\newtheorem{remark}[theorem]{Remark}

\newtheorem{definition}[theorem]{Definition}
\newtheorem{example}[theorem]{Example}

\newcommand{\C}{\mathbb{C}}
\newcommand{\D}{\mathbb{D}}
\newcommand{\B}{\mathbb{B}}

\newcommand{\R}{\mathbb{R}}

\renewcommand{\H}{\mathbb{H}}

\newcommand{\CH}{{\operatorname{CH}}}

\begin{document}
\title{On the Gromov hyperbolicity of the minimal metric}
\author{Matteo Fiacchi}
\address{Matteo Fiacchi, Faculty of Mathematics and Physics, University of Ljubljana, and  Institute of Mathematics, Physics, and Mechanics, Jadranska 19, 10000 Ljubljana, Slovenia.}
\email{matteo.fiacchi@fmf.uni-lj.si}
\thanks{The author is supported by the European Union (ERC Advanced grant HPDR, 101053085 to Franc Forstneri\v c) and the research program P1-0291 from ARIS, Republic of Slovenia.}
 
\maketitle

\begin{abstract}
In this paper, we study the hyperbolicity in the sense of Gromov of domains in $\R^d$ $(d\geq3)$ with respect to the minimal metric introduced by Forstneri\v c and Kalaj in \cite{FK}. In particular, we prove that every bounded strongly minimally convex domain is Gromov hyperbolic and its Gromov compactification is equivalent to its Euclidean closure. Moreover, we prove that the boundary of a Gromov hyperbolic convex domain does not contain non-trivial conformal harmonic disks. Finally, we study the relation between the minimal metric and the Hilbert metric in convex domains.
\end{abstract}

\section{Introduction}

In several complex variables, the Kobayashi metric plays a fundamental role in the study of domains in $\mathbb{C}^d$ and holomorphic maps between them.

In recent years, the metric approach to complex analysis has led to important results, ranging from geometric function theory to complex dynamics \cite{braccigaussier2, BGZ, AFGG, L22, Zim22, AFGK}. One of the main problems of this method is characterizing when the metric space $(D, k_D)$ is Gromov hyperbolic.

Gromov hyperbolicity is a coarse notion of negative curvature in the context of metric spaces and is very useful for studying metric spaces where the distance function may not arise from a Riemannian metric. The Gromov hyperbolicity of the Kobayashi metric has been extensively studied in recent years \cite{NTT, Zim1, Zim2, GS, Fia, LPW}.

In real Euclidean space, Forstneri\v{c} and Kalaj \cite{FK} defined the \textit{minimal metric}, the analog of the Kobayashi metric in the theory of minimal surfaces. The two metrics share many similarities, but there are notable differences. For example, since conformal minimal surfaces are much more abundant than holomorphic curves, the minimal metric is significantly more rigid than the Kobayashi metric, making some typical techniques of complex analysis inapplicable in this context.

Given $D \subset \mathbb{R}^d$ $(d \geq 3)$, we are interested in finding necessary and sufficient conditions for the minimal distance $\rho_D$ to be Gromov hyperbolic. The first natural class of domains to consider is bounded strongly minimally convex domains (see Section \ref{BSMC} for the definition), which are the analogs in the theory of minimal surfaces of strongly pseudoconvex domains.

\begin{theorem}\label{MAIN1}
Let $D \subset \mathbb{R}^d$ $(d \geq 3)$ be a bounded strongly minimally convex domain. Then the metric space $(D, \rho_D)$ is Gromov hyperbolic. Moreover, the identity map $D \to D$ extends as a homeomorphism between the Gromov compactification $\overline{D}^G$ and the Euclidean closure $\overline{D}$.
\end{theorem}

This answers a question of Drinovec Drnov\v sek and Forstneri\v c \cite[Problem 12.7]{DrFor}.

Inspired by Balogh's and Bonk's work in strongly pseudoconvex domains \cite{BaBo}, the proof is based on estimates of the minimal metric near the boundary and the construction of the hyperbolic filling metric (\refeq{fillingmetric}).

The second main result involves the necessary conditions in convex domains. As already observed in Hilbert geometry \cite{KN, Ben} and complex geometry \cite{NTT, Zim1, Zim2, GS, ADF}, the ``flatness'' of the boundary is an obstruction to Gromov hyperbolicity.
\begin{theorem}\label{MAIN2}
	Let $D \subset \mathbb{R}^d$ $(d \geq 3)$ be a hyperbolic convex domain such that $(D, \rho_D)$ is Gromov hyperbolic. Then every conformal harmonic disk $f \colon \mathbb{D} \to \partial D$ is constant.
\end{theorem}

Note that no assumptions regarding boundedness and regularity of the boundary are necessary. The proof relies on estimates of the minimal metric within convex domains and the construction of ``fat'' quasi-geodesic triangles near the conformal harmonic disk.

Finally, we provide results concerning the Hilbert metric $h_D$ and the minimal metric $g_D$ in convex domains. In particular, we prove (Proposition \ref{gvsh}) that in any convex domain $D \subset \mathbb{R}^d$ $(d \geq 3)$ we have
\[ h_D(x, v) \leq 2 g_D(x, v), \quad x \in D, \, v \in \mathbb{R}^d \]
and we prove that in strongly convex domains, these two metrics are bilipschitz.

The paper is organized as follows: Section 2 covers the preliminaries used throughout. In Section 3, we establish boundary estimates for the minimal metric in bounded strongly minimally convex domains. Sections 4 and 5 present the proofs of Theorem \ref{MAIN1} and Theorem \ref{MAIN2}, respectively. Finally, Section 6 explores the relationship between the minimal and Hilbert metrics in convex domains.

\medskip {\bf Acknowledgements.} 
The author would like to express his gratitude to Franc Forstneri\v c for introducing him to the theory of minimal metric and Rafael B. Andrist for the interesting and helpful discussions regarding Proposition \ref{boundaryconnected}.

\section{Preliminares}
\textbf{Notations:}\begin{itemize}
	\item Let $(e_j)_{j=1,\dots,d}$ denote the canonical basis of $\R^d$.
	\item For $x\in\R^d$ let $||x||$ denote the standard Euclidean norm of $x$.
	\item For $u,v\in\R^d$ let $\langle u,v\rangle$ denote the Euclidean scalar product of $\R^d$.
	\item Let $\D:=\{x\in\R^2: ||x||<1\}$ and $\H^d:=\{(x,y)\in\R\times\R^{d-1}:x>0\}$.
	\item If $D\subsetneq\R^d$ is a domain and $x\in\R^d$ let
	$$\delta_D(x):=\inf\left\{||x-y||:y\in\partial D\right\}.$$
	\item Let $G_2(\R^d)$ denote the Grassman manifold of 2-planes in $\R^d$.
	\item Let $a,b\in\R$, let denote $a\wedge b:=\min\{a,b\}$ and $a\vee b:=\max\{a,b\}$.
\end{itemize}

\subsection{Minimal metric}

A map $f\colon\D\to \R^d$ $(d\geq2)$ is said to be \textit{conformal} if for all $\zeta\in\D$ we have
$$||f_x(\zeta)||=||f_y(\zeta)|| \ \mbox{and} \ \langle f_x(\zeta), f_y(\zeta)\rangle=0$$
where $\zeta=(x,y)$ are the coordinates of $\D\subset\R^2$. Let $D\subset\R^d$ $(d\geq2)$  be a domain and denote by $\CH(\D, D)
$ the space of conformal harmonic maps $f\colon\D\to D$.

The \textit{minimal metric} of $D$ is the function
$$g_D(x,v)=\inf\{1/r: f\in \CH(\D,D), f(0)=x, f_x(0)=rv \}, \ \ \ x\in D, v\in\R^d.$$
It turns out that $g_D$ is a Finsler metric, i.e. $g_D$ is not negative, upper-semicontinuous on $D\times\R^d$ and absolutely homogeneous
$$g_D(x,tv)=|t|g_D(x,v), \ \ \ t\in\R.$$
We can consider the intrinsic distance
$$\rho_D(x,y)=\inf_\gamma\int_0^1g_D(\gamma(t),\dot{\gamma}(t))dt, \ \ \ x,y\in D$$
where the infimum is over all piecewise $\mathcal{C}^1$ curve $\gamma\colon[0,1]\to D$ with $\gamma(0)=x$ and $\gamma(1)=y$. The function $\rho_D$ is called \textit{minimal pseudodistance} of $D$, and in general it may not be a distance function (for example $\rho_{\R^d}$ vanishes identically). For this reason we say that a domain $D\subset\R^d$ is \textit{hyperbolic} if $\rho_D$ is a distance and \textit{complete hyperbolic} if $\rho_D$ is a complete distance (in the sense of Cauchy). See \cite{DrFor} for some characterizations of hyperbolicity and complete hyperbolicity.

The minimal metric can be characterized in the following way: it is the largest pseudometric such that for every conformal surface $M$ and conformal harmonic map $f\colon M\to D$ we have
\begin{equation}\label{kconformal}
g_D(f(z),df_z(v))\leq\kappa_M(z,v), \ \ \ \ z\in M, v\in T_zM,
\end{equation}
where $\kappa_M$ is the hyperbolic metric of $M$.

Calculating the minimal metric of a domain explicitly is a challenging task. Forstneri\v{c} and Kalaj in \cite{FK} managed to compute it in the Euclidean unit ball of $\R^d$ $(d\geq3)$, showing that it equals the \textit{Beltrami-Cayley metric}.

\begin{example}\label{exball}
For $d\geq3$, let $x\in\B^d$ and $v\in\R^d$, then 
$$g_{\B^d}(x,v)=\dfrac{(1-||x||^2)||v||^2+|\langle x, v\rangle|^2}{(1-||x||^2)^2}=\dfrac{||v||^2}{1-||x||^2}+\dfrac{|\langle x, v\rangle|^2}{(1-||x||^2)^2}.$$
\end{example}

We conclude this section by noting that if $R\colon\R^d\to\R^d$ is a rigid transformation (i.e., a composition of an orthogonal map, a dilation and a translation) then for all $D\subset\R^d$, the map $R\colon(D,\rho_D)\to (R(D),\rho_{R(D)})$ is an isometry. In other words, the minimal distance is invariant under rigid transformations. These transformations are the only ones in $\R^d$ with this property.

\subsection{Strongly minimally convex domains}\label{BSMC}

Let $U\subset\R^d$ be a open set and $u\colon U\to\R$ be a $\mathcal{C}^2$-smooth function. Given $x\in U$ and a 2-plane $\Lambda\in G_2(\R^d)$, let
$$tr_\Lambda Hess_u(x)$$
denote the trace of the restriction of the Hessian of $u$ to $\Lambda$.

\begin{definition}
 A $\mathcal{C}^2$-smooth function $u\colon U\to\R$ is \textit{minimal plurisubharmonic} (MPSH) if for all $x\in U$ and $\Lambda\in G_2(\R^d)$ we have
 \begin{equation}\label{hess}tr_\Lambda Hess_u(x)\geq0,
\end{equation} that is equivalent to say that for all $x\in U$ the smallest two eigenvalues $\lambda_1(x)$ and $\lambda_2(x)$ of $Hess_u(x)$ satisfies $\lambda_1(x)+\lambda_2(x)\geq0$.
Moreover, we say that $u$ is \textit{strongly minimal plurisubharmonic} if strong inequality holds in (\refeq{hess}).
\end{definition}

\begin{definition}
A bounded domain $D\subset\R^d$ $(d\geq3)$ with $\mathcal{C}^2$ boundary is \textit{strongly minimally convex} if it admits a $\mathcal{C}^2$ defining function, that is a function $u:U\to\R$ with $D=\{x\in U: u(x)<0\}$ and $du\neq0$ on $\partial D=\{x\in U: u(x)=0\}$, which is strongly MPSH on a neighborhood of $\partial D$.
\end{definition}

Bounded strongly minimal convex domains can be seen as analogous to strongly pseudoconvex domains in minimal metric theory.

\begin{remark}\label{rmknu}
A domain $D\subset\R^d$ $(d\geq3)$ admits a (strongly) MPSH defining function if and only if for all $\xi\in\partial D$ the interior principal curvatures $\nu_1\leq\nu_2\leq\cdots\leq\nu_{d-1}$ of $\partial D$ satisfy $\nu_1+\nu_2\geq0$ $(\nu_1+\nu_2>0)$ (see \cite[Theorem 8.1.13]{AFLbook}).
\end{remark}

The previous remark readily implies that strongly minimally convex domains have a connected boundary.

\begin{proposition}\label{boundaryconnected}
Let $D\subset\R^3$ $(d\geq3)$ be a bounded domain with $\mathcal{C}^2$ boundary and suppose that admits a MPSH defining function, then $\partial D$ is connected.
\end{proposition}
\proof
By contradiction, if $\partial D$ is not connected, then we can write $D$ as $D_1\backslash \overline{D}_2$, where $D_2\subset D_1$ are two bounded domains (see, for example, \cite{CKL}). Now let $\xi$ be a point in $\partial D_2$ furthest from the origin. At $\xi$, all the interior principal curvatures of $\partial D$ are negative, contradicting Remark \ref{rmknu}.
\endproof

We conclude this section recalling the following theorem proved in \cite[Theorem 9.2]{DrFor}.

\begin{theorem}\label{complBSMC}
If  $D\subset\R^d$ $(d\geq3)$ is a strongly minimally convex domain, then it is complete hyperbolic.
\end{theorem}

\subsection{Gromov hyperbolic spaces}
The book \cite{BH} is a standard reference for the theory of Gromov hyperbolic spaces.

\begin{definition} 
	Let $(X,d)$ be a metric space.
	For every $x,y,o\in X$  the {\sl Gromov product} is defined as
	$$(x|y)_o:=\frac{1}{2}[d(x,o)+d(y,o)-d(x,y)].$$
Let $\delta\geq0$. The metric space $(X,d)$ is $\delta$-{\sl hyperbolic} if for all $x,y,z,o\in X$
	\begin{equation}\label{4point}
(x|y)_o\geq(x|z)_o\wedge(y|z)_o-\delta.
	\end{equation}
	Finally, a metric space is {\sl Gromov hyperbolic} if it is $\delta$-hyperbolic for some $\delta\geq0$.
\end{definition}

\begin{remark}\label{4condII}
The inequality (\refeq{4point}) is equivalent to
$$d(x,z)+d(y,w)\leq(d(x,y)+d(z,w))\vee(d(y,z)+d(x,w))+2\delta, \ \ \ \ x,y,z,w\in X.$$
\end{remark}

\begin{definition}\label{Gcomp}[Gromov boundary]
Fix a base-point $o\in X$. A sequence $(x_i)$ in $X$ \textit{converges at infinity} if $\lim_{i,j\to+\infty}(x_i|x_j)_o=+\infty$. Two sequences $(x_i)$ and $(y_i)$ converging at infinity are
called \textit{equivalent} if $\lim_{i\to+\infty}(x_i|y_i)_o =+\infty$. By (\refeq{4point}), this defines an equivalent relation. It is easy to see that these notions do not depend on the choice of $o\in X$.

The \textit{Gromov boundary} $\partial_G X$ is defined as the set of equivalence classes of sequences converging at infinity. The \textit{Gromov compactification} $\overline{X}^G:=X\sqcup\partial_GX$ can be equipped with a topology that provides a compactification of the space $X$.
\end{definition}

\begin{definition}
	Let $(X,d)$ be a metric space, $I\subset\R$ be an interval and $A\geq1$ and $B\geq0$. A function $\sigma\colon I\rightarrow X$ is
	\begin{enumerate}
		\item a {\sl geodesic}, if for each $s,t\in I$ 
		$$d(\sigma(s),\sigma(t))=|t-s|; $$
		\item a $(A,B)$-{\sl quasi-geodesic}, if for each $s,t\in I$
		$$A^{-1}|t-s|-B\leq d(\sigma(s),\sigma(t))\leq A|t-s|+B.$$
	\end{enumerate}
A $(A,B)$-\textit{quasi-geodesic triangle} consists of three points in $X$ and three $(A,B)$-quasi-geodesic segments connecting these points, known as its \textit{sides}. If $M\geq0$, a quasi-geodesic triangle is $M$-\textit{slim} if each side is contained within the $M$-neighborhood of the other two sides.
\end{definition}

Recall that a metric space $(X,d)$ is \textit{proper} if closed balls are compact, and \textit{geodesic} if any two points can be connected by a geodesic. A fundamental property of proper geodesic Gromov hyperbolic spaces is the Shadowing lemma \cite[Part III, Theorem 1.7]{BH}, which implies the following characterization of Gromov hyperbolicity.

\begin{proposition}\cite[Part III, Corollary 1.8]{BH}\label{thinqtriangle}
	A proper geodesic metric space $(X, d)$ is $\delta$-hyperbolic if and only if, for all $A\geq1$ and $B\geq0$, there exists $M:=M(A,B,\delta)>0$ such that every $(A,B)$-quasi-geodesic triangle is $M$-slim.
\end{proposition}

Moreover, the Shadowing Lemma implies the invariance of Gromov hyperbolicity under bilipschitz maps.

\begin{corollary}\cite[Part III, Theorem 1.9]{BH}\label{Gromovinv}
Let $(X,d_X)$ and $(Y,d_Y)$ be two metric spaces, and $f\colon X\to Y$ a bijective bilipschitz map. Then $(X,d_X)$ is Gromov hyperbolic if and only if $(Y,d_Y)$ is Gromov hyperbolic. Moreover, $f$ extends to a homeomorphism between the Gromov compactifications $\overline{X}^G$ and $\overline{Y}^G$.
\end{corollary}

\section{Estimates of the minimal metric in Bounded Strongly Minimally Convex domains}

We begin a well-known lemma on boundary projection.

\begin{lemma}\cite[Lemma 2.1]{BaBo}\label{lemma2.1}
Let $D\subset\R^d$ $(d\geq2)$ be a bounded domain with $\mathcal{C}^2$ boundary. Then there exists $0<\epsilon<1$ such that
\begin{enumerate}
\item for every $x\in N_\epsilon(\partial D):=\{x\in\R^d: \delta_D(x)<\epsilon\}$ there exists a unique $\pi(x)\in\partial D$ with $||x-\pi(x)||=\delta_D(x)$;
\item the \textit{signed distance} to the boundary $\rho\colon \R^d\to\R$ given by
$$\rho(x):=\begin{cases*}
	-\delta_D(x) &\text{if\ $x\in D$}\\
	\delta_D(x) &\text{if\ $x\notin D$}
	\end{cases*}$$is $\mathcal{C}^2$-smooth on $N_\epsilon(\partial D)$;
\item the fibers of the map $\pi\colon N_\epsilon(\partial D)\to \partial D$ are
$$\pi^{-1}(\xi)=\{\xi+tn_\xi: |t|<\epsilon\}$$
where $n_\xi$ is the outer unit normal vector of $\partial D$ at $\xi\in\partial D$;
\item the gradient of $\rho$ satisfies  for all $x\in N_\epsilon(\partial D)$
$$\nabla\rho(x)=n_{\pi(x)};$$
\item the projection map $\pi\colon N_\epsilon(\partial D)\to \partial D$ is $\mathcal{C}^1$-smooth. 
\end{enumerate}
\end{lemma}

Let $D\subset\R^d$ $(d\geq2)$ be a bounded domain with $\mathcal{C}^2$ boundary, and let $0<\epsilon<1$ be the constant from the previous Lemma. For $x\in N_\epsilon(\partial D)$ and $v\in\R^d$, we consider the orthogonal decomposition of $v=v_N+v_T$ at (unique) projection point $\pi(x)$, where
$$v_N:=\langle v,n_{\pi(x)}\rangle n_{\pi(x)}, \ \ \ \  v_T:= v-v_N.$$

We are now ready to state the main result of this section.

\begin{theorem}\label{BSMCestimate}
Let $D\subset\R^d$ $(d\geq3)$ be a bounded strongly minimally convex domain. Then there exists $0<\epsilon<1$ and $A\geq1$ such that
for all $x\in N_\epsilon(\partial D)\cap D$ and $v\in\R^d$, we have
$$A^{-1}\left(\frac{||v_N||}{2\delta_D(x)}+\frac{||v_T||}{\delta_D(x)^{1/2}} \right)\leq g_D(x,v)\leq A\left(\frac{||v_N||}{2\delta_D(x)}+\frac{||v_T||}{\delta_D(x)^{1/2}} \right)$$
\end{theorem}

To prove the theorem, we begin by examining the minimal metric of the Euclidean ball.

\begin{lemma}\label{estball} For all $x\in\B^d$, $x\neq 0$ and $v\in\R^d$ we have
$$g_{\B^d}(x,v)\leq2\left(\frac{||v_N||}{1-||x||}+\frac{||v_T||}{(1-||x||)^{1/2}}\right) .$$
\end{lemma}
\proof
By the explicit formula in Example \ref{exball} we have
\begin{align*}g_{\B^d}(x,v)&=\left(\frac{||v||^2}{1-||x||^2}+\frac{|\langle v, x\rangle|^2}{(1-||x||^2)^2}\right)^\frac{1}{2}\\&=\left(\frac{||v_N||^2+||v_T||^2}{1-||x||^2}+\frac{||x||^2||v_N||^2}{(1-||x||^2)^2}\right)^\frac{1}{2}\\&\leq \frac{||v_N||+||v_T||}{\sqrt{1-||x||^2}}+\frac{||v_N||}{1-||x||^2}\\&\leq\frac{||v_T||}{\sqrt{1-||x||^2}}+2\frac{||v_N||}{1-||x||^2}\\&=2\left(\frac{||v_N||}{2(1-||x||)}+\frac{||v_T||}{(1-||x||)^{1/2}}\right).
\end{align*}
\endproof

\begin{proof}[Proof of Theorem \ref{BSMCestimate}]
\textbf{Upper bound:} Let $x\in N_\epsilon(\partial D)\cap D$ and consider $o_x:=\pi(x)-\epsilon n_{\pi(x)}.$ Notice that $B(o_x,\epsilon)\subset D$. Using Lemma \ref{estball} and noticing that $||x-o_x||=\epsilon-\delta_D(x)$, we have \begin{align*}g_D(x,v)&\leq g_{B(o_x,\epsilon)}(x,v)\\&=\epsilon^{-1}g_{\B^d}(\epsilon^{-1}(x-o_x),v)\\&\leq 2\epsilon^{-1}\left(\epsilon\frac{||v_N||}{2\delta_D(x)}+\epsilon^{1/2}\frac{||v_T||}{\delta_D(x)^{1/2}}\right)\\&\leq 2\epsilon^{-1/2}\left(\frac{||v_N||}{2\delta_D(x)}+\frac{||v_T||}{\delta_D(x)^{1/2}}\right).\end{align*}
\textbf{Lower Bound:}
By (9.7) in \cite{DrFor}, there exists $c_1>0$ such that
$$g_D(x,v)\geq c_1\frac{||v_N||}{\delta_D(x)}.$$
Moreover, by \cite[Corollary 9.1]{DrFor}, there exists $c_2>0$ such that
$$g_D(x,v)\geq c_2\frac{||v||}{\delta_D(x)^{1/2}}\geq c_2\frac{||v_T||}{\delta_D(x)^{1/2}},$$
so
$$g_D(x,v)\geq c\left(\frac{||v_N||}{2\delta_D(x)}+\frac{||v_T||}{\delta_D(x)^{1/2}} \right),$$
where $c=\min\{c_1,c_2\}/2$.
\end{proof}

\section{The Gromov hyperbolicity of Bounded Strongly Minimally Convex domains}

In this section, we will prove the Gromov hyperbolicity of bounded strongly minimally convex domains with respect to the minimal metric (Theorem \ref{MAIN1}). We start with following auxiliary result.

\begin{lemma}\label{gammavsalpha}
Let $D\subset\R^d$ $(d\geq2)$ be a bounded domain with $\mathcal{C}^2$ boundary. Then there exists $0<\epsilon<1$ such that if $\gamma\colon [0,1]\to \overline{N_\epsilon(\partial D)}$ is a $\mathcal{C}^1$-smooth curve and $\alpha:=\pi\circ\gamma\colon [0,1]\to\partial D$ its projection to the boundary, then for all $t\in[0,1]$
$$\frac{1}{2}||\dot{\gamma}_T(t)||\leq||\dot{\alpha}(t)||\leq 2||\dot{\gamma}_T(t)||.$$
\end{lemma}
\proof
Let $0<\epsilon_0<1$ be the constant of Lemma \ref{lemma2.1}.
Now, let $\gamma\colon [0,1]\to N_\epsilon(\partial D)$ be a generic $\mathcal{C}^1$-smooth curve with $\epsilon\leq\epsilon_0$ to determine. 
We have 
$$\gamma(t)-\alpha(t)=-\delta_D(\gamma(t))n_{\alpha(t)}.$$
Differentiating we obtain
$$\dot{\gamma}(t)-\dot{\alpha}(t)=\delta_D(\gamma(t))M(t)\dot{\alpha}(t)+\langle a(t),\dot{\gamma}(t)\rangle n_{\alpha(t)}$$
where $M$ is real ($d\times d$)-matrix valued functions and $a$ is a vector valued function.
The matrix $M$  can be expressed in terms of the first and second derivatives of $\rho$ evaluated at
points in  $N_{\epsilon_0}(\partial D)$, ensuring that $||M||$ is uniformly bounded. Considering the tangential component at  $\alpha(t)$, there exists $C>0$ such that
$$||\dot{\gamma}_T(t)-\dot{\alpha}(t)||\leq C\delta_D(\gamma(t))||\dot{\alpha}(t)||$$
which implies 
$$||\dot{\gamma}_T(t)||\leq (1+C\delta_D(\gamma(t)))||\dot{\alpha}(t)||$$
and
$$ (1-C\delta_D(\gamma(t)))||\dot{\alpha}(t)||\leq||\dot{\gamma}_T(t)||.$$
We obtain the desired estimates by setting $\epsilon:=\min\{\epsilon_0,C^{-1}\}/2$.
\endproof

Let $D\subset\R^d$ $(d\geq3)$ be a bounded strongly minimally convex domain. For the rest of the section, we denote $\overline{N_\epsilon(\partial D)}\cap D$ by $N$, where $0<\epsilon<1$ is the constant from the previous Lemma.

Firstly, we recall that by Proposition \ref{boundaryconnected}, the boundary $\partial D$ is connected. Therefore, we can consider the intrinsic distance on the boundary 
 $H\colon \partial D\times \partial D\to [0,+\infty)$ defined as
$$H(\xi,\eta):=\inf\left\{\int_0^1||\dot{\alpha}(t)||dt:\alpha\colon [0,1]\to \partial D, \mbox{piecewise} \ \mathcal{C}^1, \alpha(0)=\xi, \alpha(1)=\eta\right\}$$
and $d_H\colon D\times D\to[0,+\infty)$ given by
\begin{equation}\label{fillingmetric}
d_H(x,y)=2\log\left(\frac{H(\pi(x),\pi(y))+h(x)\vee h(y)}{\sqrt{h(x)h(y)}}\right),
\end{equation}
where $h(x):=\delta_D(x)^{1/2}$ and $\pi\colon D\to \partial D$ is a function such that $||x-\pi(x)||=\delta_D(x)$. The function $\pi(x)$ is uniquely determined only in a tubular neighborhood of $\partial D$ (Lemma \ref{lemma2.1}), but the choice of $\pi$ will not affect the final result. The function $d_H$ turns out to be a distance function on $D$ (see \cite[Lemma 7.1]{BS}).

Now, we define the function $F\colon D\times\R^d\to [0,+\infty)$ as follows
$$F(x,v)=\begin{cases*}
\dfrac{||v_N||}{2\delta_D(x)}+\dfrac{||v_T||}{\delta_D(x)^{1/2}} &\text{if \ $x\in N, v\in\R^d$} \\
c||v|| &\text{if \ $x\in D\backslash N, v\in\R^d$}
\end{cases*}$$
where $c>0$ is chosen so that $F$ is upper-semicontinuous. It is straightforward to see that 
$F$ defines a Finsler metric.

If $\gamma\colon [0,1]\to D$ is a piecewise $\mathcal{C}^1$ curve, we define
$$\ell_F(\gamma):=\int_0^1F(\gamma(t),\dot{\gamma}(t))dt$$
the length of $\gamma$ with respect to the Finsler metric $F$.
Finally, we define $d_F\colon D\times D\to[0,+\infty)$ as the intrinsic distance
$$d_F(x,y): =\inf_\gamma\ell_F(\gamma), \ \ \ x,y\in D,$$
where the infimum is over all piecewise $\mathcal{C}^1$ curve $\gamma\colon[0,1]\to D$ with $\gamma(0)=x$ and $\gamma(1)=y$. Note that Theorem \ref{BSMCestimate} implies that $d_F$ and $\rho_D$ are bilipschitz, meaning there exists $A\geq1$ such that for all $x,y\in D$
\begin{equation}\label{bilip}
A^{-1}d_F(x,y)\leq\rho_D(x,y)\leq Ad_F(x,y).
\end{equation}
In particular, by Theorem \ref{complBSMC} $(D,d_F)$ is a complete metric space.

\begin{lemma}\label{curesti1}
Let $\gamma\colon [0,1]\to N$ be a piecewise $\mathcal{C}^1$ curve with endpoints $x,y\in N$, then
$$\ell_F(\gamma)\geq \left|\log\frac{h(y)}{h(x)}\right|+\frac{1}{2}\frac{H(\pi(x),\pi(y))}{h_\gamma},$$
where $h_\gamma:=\max_{t\in[0,1]}h(\gamma(t))$.
Moreover, if $\gamma$ is a straight line contained in $\pi^{-1}(\xi)$ with $\xi\in\partial D$, then
$$\ell_F(\gamma)=\left|\log\frac{h(y)}{h(x)}\right|.$$
\end{lemma}
\proof
First of all, for those $t\in[0,1]$ for which $\dot\gamma(t)$ exists, $$\left|\frac{d}{dt}\log(\delta_D(\gamma(t)))\right|_{t=t_0}=\frac{||(\dot{\gamma}_N)(t_0)||}{\delta_D(\gamma(t_0))}.$$
So by Lemma \ref{gammavsalpha},
\begin{align*}\ell_F(\gamma)&= \int_0^1\frac{||(\dot{\gamma}(t))_N||}{2\delta_D(\gamma(t))}+\frac{||(\dot{\gamma}(t))_T||}{\delta_D(\gamma(t))^{1/2}}dt
\\&\geq\frac{1}{2}\left|\log\frac{\delta_D(y)}{\delta_D(x)}\right|+\frac{1}{2}\int_0^1\frac{||\dot{\alpha}(t)||}{\delta_D(\gamma(t))^{1/2}}dt\\&\geq\left|\log\frac{h(y)}{h(x)}\right|+\frac{1}{2}\frac{H(\pi(x),H(y))}{h_\gamma}.
\end{align*}
Finally, if $\gamma$ is a straight line, then $\dot{\gamma}_T\equiv0$, so
$$\ell_F(\gamma)= \int_0^1\frac{|(\dot{\gamma}(t))_N|}{2\delta_D(\gamma(t))}dt=\frac{1}{2}\left|\log\frac{\delta_D(y)}{\delta_D(x)}\right|=\left|\log\frac{h(y)}{h(x)}\right|.$$
\endproof

We can now prove the main result of this section.

\begin{theorem}\label{Bquasiisome}
Let $D\subset\R^d$ $(d\geq3)$ be a bounded domain with $\mathcal{C}^2$ boundary, then there exists $B\geq0$ such that for all $x,y\in D$
\begin{equation}\label{eqBesti}
d_H(x,y)-B\leq d_F(x,y)\leq d_H(x,y)+B.
\end{equation}
\end{theorem}
\proof 

In order to prove the two estimates, we consider various cases depending on the relative position of $x$ and $y$. To simplify the notation, we denote with $B$ and $B'$ all the additive constants that can vary case by case.\medskip 

\textbf{Case 1:} $x,y\in D\backslash N$.

Since $D\backslash N$ is relative compact in $D$, both functions $d_H$ and $d_F$ are bounded in $D\backslash N\times D\backslash N$ so (\refeq{eqBesti}) is trivial.\medskip

\textbf{Case 2:} $x\in N$, $y\in D\backslash N$ or $x\in D\backslash N$, $y\in N$.

We may assume $x\in N$ and $y\in D\backslash N$. From the definition of $d_H$, it follows that there exists $B>0$ such that
$$\log\left(\frac{1}{h(x)}\right)-B\leq d_H(x,y)\leq \log\left(\frac{1}{h(x)}\right)+ B.$$
Set $x'=\pi(x)+\epsilon n_{\pi(x)}$. Then
$$d_F(x,y)\leq d_F(x,x')+d_F(x',y).$$
By Lemma \ref{curesti1} $d_F(x,x')=\log\left(\frac{\epsilon^{1/2}}{h(x)}\right)$. Since $x'\in \overline{D\backslash N}\subset\subset D$, we can find $B'>0$ independent of $x'$ and $y$ such that $d_F(x',y)\leq B'$, so we obtain the upper bound 
$$d_F(x,y)\leq \log\left(\frac{1}{h(x)}\right)+\frac{1}{2}\log\epsilon+ B'.$$
For the lower bound, let $\gamma\colon [0,1]\to D$ be a piecewise $\mathcal{C}^1$ curve starting from $x$ with endpoint $y$ such that $\ell_F(\gamma)\leq d_F(x,y)+1$. Let 
$$t^*:=\inf\{t\in[0,T]:\delta_D(\gamma(t))=\epsilon\}.$$
By Lemma \ref{curesti1} we obtain $$d_F(x,y)\geq\ell_F(\gamma)-1\geq\ell_F(\gamma|_{[0,t^*]})-1\geq\log\left(\frac{\epsilon^{1/2}}{h(x)}\right)-1=\log\left(\frac{1}{h(x)}\right)+\frac{1}{2}\log\epsilon-1.$$

Now set $\xi:=\pi(x)$ and $\eta:=\pi(y)$.\medskip

\textbf{Case 3:} $x,y\in N, h(x)\vee h(y)\geq H(\xi,\eta)$.

We may assume $h(y)\geq h(x)$. From the definition of $d_H$, there exists $B>0$ such that
$$\log\left(\frac{h(y)}{h(x)}\right)\leq d_H(x,y)\leq\log\left(\frac{h(y)}{h(x)}\right)+ B.$$
By Lemma \ref{curesti1} we have the lower  bound $$d_F(x,y)\geq \log\left(\frac{h(y)}{h(x)}\right).$$

For the upper bound, consider $x':=\xi+h^2(y)n_\xi$. Notice that $\delta_D(x')=\delta_D(y)$. Let $\alpha\colon [0,1]\to \partial D$ be a geodesic with respect to $H$ connecting $\xi$ with $\eta$, and define $\gamma(t):=\alpha(t)+h(y)^2n_{\alpha(t)}$. Notice that $\gamma$ is a curve in $D$ connecting $x'$ to $y$ such that $\delta_D(\gamma(t))\equiv \delta_D(y)$. Clearly $\dot{\gamma}(t)_N\equiv0$, so by Lemma \ref{gammavsalpha}
$$d_F(x',y)\leq \ell_F(\gamma)=\int_0^1\frac{||\dot{\gamma}_T(t)||}{h(y)}dt\leq \frac{2}{h(y)}\int_0^1||\dot{\alpha}(t)||dt=2\frac{H(\xi,\eta)}{h(y)}\leq 2.$$ Finally, using again Lemma \ref{curesti1} we have $$d_F(x,y)\leq d_F(x,x')+d_F(x',y)\leq\log\left(\frac{h(y)}{h(x)}\right)+2.$$

\textbf{Case 4:}  $x,y\in N, h(x)\vee h(y)\leq H(\xi,\eta)$.

Assume again $h(y)\geq h(x)$. Then
$$2\log\left(\frac{H(\xi,\eta)}{\sqrt{h(x)h(y)}}\right)\leq d_H(x,y)\leq2\log\left(\frac{H(\xi,\eta)}{\sqrt{h(x)h(y)}}\right)+ 2\log2.$$
Set $\hat{H}(\xi,\eta):=\min\{H(\xi,\eta),\epsilon^{1/2}\}$. Notice that the function $H(\cdot,\cdot)$ is bounded from above, so we can find $A>1$, independent of $\xi$ and $\eta$, such that $\hat{H}(\xi,\eta)\leq H(\xi,\eta)\leq A\hat{H}(\xi,\eta)$. Consider the points $x':=\xi+\hat{H}(\xi,\eta)^2n_\xi$ and $y':=\eta+\hat{H}(\xi,\eta)^2n_\eta$, then reasoning as in Case 3, we have
$$d_F(x',y')\leq 2\frac{H(\xi,\eta)}{\hat{H}(\xi,\eta)}\leq 2A.$$ Finally, by Lemma \ref{curesti1} we have
\begin{align*}d_F(x,y)&\leq d_F(x,x')+d_F(x',y')+d_F(y',y)\\&\leq\log\left(\frac{\hat{H}(\xi,\eta)}{h(x)}\right)+2A+\log\left(\frac{\hat{H}(\xi,\eta)}{h(x)}\right)\\&= 2\log\left(\frac{H(\xi,\eta)}{\sqrt{h(x)h(y)}}\right)+ 2A,
\end{align*}
that is the upper bound.

For lower bound, let $\gamma\colon [0,1]\to D$ be a piecewise $\mathcal{C}^1$ curve joining $x$ and $y$. Define $h_\gamma:=\max_{t\in[0,1]}h(\gamma(t))$. In order to estimate $\ell_F(\gamma)$ from below, we need to consider two sub-cases:\medskip

\textbf{Case 4.1:} $h_\gamma\geq\hat{H}(\xi,\eta)$.

Define $t_1:=\inf\{t\in[0,1]:h(\gamma(t))=\hat{H}(\xi,\eta)\}$ and $t_2:=\sup\{t\in[0,1]:h(\gamma(t))=\hat{H}(\xi,\eta)\}$. By Lemma \ref{curesti1} \begin{align*}\ell_F(\gamma)&\geq\ell_F(\gamma|_{[0,t_1]})+\ell_F(\gamma|_{[t_2,1]})\\&\geq \log\left(\frac{\hat{H}(\xi,\eta)}{h(x)}\right)+\log\left(\frac{\hat{H}(\xi,\eta)}{h(y)}\right)\\&\geq2\log\left(\frac{\hat{H}(\xi,\eta)}{\sqrt{h(x)h(y)}}\right)\\&\geq2\log\left(\frac{H(\xi,\eta)}{\sqrt{h(x)h(y)}}\right)-2\log A.\end{align*}\medskip

\textbf{Case 4.2:} $h_\gamma<\hat{H}(\xi,\eta)$.

Let $t_0\in[0,1]$ such that $h(\gamma(t_0))=h_\gamma$. Then by Lemma \ref{curesti1} 
\begin{align*}\ell_F(\gamma)&=\ell_F(\gamma|_{[0,t_0]})+\ell_F(\gamma|_{[t_0,1]})\\&\geq \log\left(\frac{h_\gamma}{h(x)}\right)+\frac{1}{2}\frac{H(\xi,\pi(\gamma(t_0)))}{h_\gamma}+\log\left(\frac{h_\gamma}{h(y)}\right)+\frac{1}{2}\frac{H(\pi(\gamma(t_0)),\eta)}{h_\gamma}\\&\geq2\log\left(\frac{h_\gamma}{\sqrt{h(x)h(y)}}\right)+\frac{1}{2}\frac{H(\xi,\eta)}{h_\gamma}.
\end{align*}
Now consider the function $f\colon (0,+\infty)\to \R$ given by $$f(u):=2\log\left(\frac{u}{\sqrt{h(x)h(y)}}\right)+\frac{1}{2}\frac{H(\xi,\eta)}{u}.$$ A simple computation shows that $f$ has a global minimum at $u=H(\xi,\eta)/4$, so
$$\ell_F(\gamma)\geq 2\log\left(\frac{H(\xi,\eta)}{\sqrt{h(x)h(y)}}\right)-4\log2+2.$$
So in both cases if we take the infimum over all piecewise $\mathcal{C}^1$ curves connecting $x$ with $y$ we obtain the lower bound
$$d_F(x,y)\geq 2\log\left(\frac{H(\xi,\eta)}{\sqrt{h(x)h(y)}}\right)-B',$$
where $B'=\max\{2\log A,4\log2-2\}$.

We considered all the possibilities for $x$ and $y$, so we proved (\refeq{eqBesti}).
\endproof

\begin{proposition}\label{FGromov}
The metric space $(D, d_F)$ is Gromov hyperbolic and the identity map $D\to D$ extends as a homeomorphism between the Gromov compactification $\overline{D}^G$ of $(D,d_F)$ and $\overline{D}$.
\end{proposition}
\proof
The proof of Gromov hyperbolicity is closely analogous to the proofs found in \cite[Theorem 7.2]{BS} and  \cite[Theorem 1.4]{BaBo}. Nevertheless, we will briefly recapitulate the argument here.

Let $x_1,x_2,x_3,x_4\in D$ be four arbitrary points of $D$ and for all $i,j\in\{1,2,3,4 \}$ set $r_{i,j}:=H(\pi(x_i),\pi(x_j))+h(x_i)\vee h(x_j)$. By the triangular inequality of $H$, we have
$$r_{1,2}r_{3,4}\leq4[(r_{1,3}r_{2,4})\vee(r_{1,4}r_{2,3})]$$
which implies
$$d_H(x_1,x_2)+d_H(x_3,x_4)\leq (d_H(x_1,x_3)+d_H(x_2,x_4))\vee(d_H(x_1,x_4)+d_H(x_2,x_3))+4\log2.$$
Finally, by Theorem \ref{Bquasiisome} we have
$$d_F(x_1,x_2)+d_F(x_3,x_4)\leq (d_F(x_1,x_3)+d_F(x_2,x_4))\vee(d_F(x_1,x_4)+d_F(x_2,x_3))+4\log2+4B$$
which implies by Remark \ref{4condII} that $(D,d_F)$ is Gromov hyperbolic.

Fix $o\in D$. Another straightforward computation, using Theorem \ref{Bquasiisome}, shows that there exists $B'>0$ such that for all $x,y\in D$,
$$\left|(x|y)_o^F+\log[H(\pi(x),\pi(y))+h(x)\vee h(y)]\right|\leq B',$$
where $(\cdot|\cdot)_o^F$ denotes the Gromov product with respect to $d_F$.
This implies that if $(x_n)$ and $(y_n)$ are two sequences in $D$ converging to infinity (Definion \ref{Gcomp}), then
$$\lim_{n\to+\infty}(x_n|y_n)_o^F=+\infty$$
if and only there exists $\xi\in\partial D$ such that $x_n\to\xi$ and $y_n\to\xi$. We conclude by reasoning similar to that in \cite[Proposition 1.2]{Zim1} or \cite[Theorem 3.3]{BNT}.
\endproof

Finally, Theorem \ref{MAIN1} easy follows from (\refeq{bilip}) and Corollary \ref{Gromovinv}.

\section{A Necessary condition for convex domains}

In this section we prove Theorem \ref{MAIN2}.

We start with the characterization of the hyperbolicity in convex domains.

\begin{theorem}\label{thDrFor}\cite[Theorem 5.1]{DrFor}
Let $D\subset\R^d$ $(d\geq3)$ be a convex domain, then the following conditions are equivalent.
\begin{enumerate}
\item The domain $D$ is hyperbolic;
\item The domain $D$ is complete hyperbolic;
\item The domain $D$ does not contain any 2-dimensional affine subspaces.
\end{enumerate}
\end{theorem}

First, we prove the following reduction result: if there exists a non-trivial conformal harmonic disk on the boundary, then there is also an affine one.

\begin{proposition}\label{affinediskboundary}
Let $D\subset\R^d$ $(d\geq3)$ be a convex domain and $f\colon \D\to \partial D$ be a non-constant conformal harmonic function. Then there exists an affine 2-plane $L$ such that $\partial D\cap L$ has non-empty interior in $L$.
\end{proposition}
\proof
Let $\Delta:=f(\D)$.
By a rigid change of coordinates, we can assume that $D\subseteq \H^d$ and the origin of $\R^d$ is contained in $\Delta$. Write $f=(f_1,f_2)\in\R\times\R^{d-1}$. Since $f$ is harmonic, by the maximum principle we have $f_1\equiv0$, so $\Delta\subset\partial D\cap\{x=0\}$.

Let $\hat{\Delta}$ denote the convex hull of $\Delta$. By the convexity of $D$, we have $\hat{\Delta}\subset\partial D\cap\{x=0\}$. 
Finally, since $f$ is conformal, $\hat{\Delta}$ cannot be of dimension $1$. This implies the statement.\endproof

Now we calculate the minimal metric of the half-space $\H^d$.

\begin{lemma}\label{halfp}
Let $\H^d$ $(d\geq3)$ be the half-space, then
$$g_{\H^d}((x,y),(u,v))=\frac{|u|}{2x}$$
for all $(x,y)\in\H^d$ and $(u,v)\in\R\times\R^{d-1}$. \end{lemma}
\proof
\textbf{Lower bound:} The estimate has already been proven in the proof of \cite[Lemma 5.2]{DrFor}, but we will repeat the argument for completeness.

Let $(x,y)\in\H^d$ and $(u,v)\in\R^d$. Suppose $f=(f_1,f_2)\colon \D\to \H^d$ is a conformal harmonic disk such that $f(0)=(x,y)$ and $f_x(0)=r(u,v)$ for some $r>0$. Since $f_1\colon \D\to[0,+\infty)$ is a positive harmonic function with $f_1(0)=x$ and $(f_1)_x(0)=ru$, by the Schwarz lemma we have
$$r|u|=|(f_1)_x(0)|\leq|\nabla f_1(0)|\leq2|f_1(0)|=2x,$$
which implies
$$\frac{1}{r}\geq\frac{|u|}{2x}.$$
Therefore, by definition
$$g_{\H^d}((x,y),(u,v))\geq\frac{|u|}{2x}.$$

\textbf{Upper bound:}
Let $(x,y)\in\H^d$ and $(u,v)\in\R^d$.

If $u=0$, it clear that $g_{\H^d}((x,y),(u,v))=0$.

If $u\neq0$, since the minimal metric is absolutely homogeneous, we can suppose that $u>0$. Let $w\in\R^{d-1}$ such that $\langle v,w\rangle=0$ and $||w||^2=u^2+||v||^2$. Now consider the affine map 
$f\colon \H^2\to\R^d$ given by
$$f(a,b)=\left(au, y-\frac{x}{u}+av+bw\right).$$
Notice that $f$ is conformal,
and $f(\H^2)\subset\H^d$. We recall that the hyperbolic metric $\kappa_{\H^2}$ of $\H^2$ is
$$\kappa_{\H^2}((a,b),\nu)=\frac{||\nu||}{2a}, \ \ \ (a,b)\in\H^2,\nu\in\R^2.$$ Since $f\left(\frac{x}{u},0\right)=(x,y)$ and $df(e_1)=(u,v)$, by (\refeq{kconformal})
$$g_{\H^d}((x,y),(u,v))\leq\kappa_{\H^2}\left(\frac{x}{u},e_1\right)=\frac{u}{2x}.$$
\endproof

We now prove a lower bound for the minimal distance in hyperbolic convex domains. Note the analogies with the complex case \cite[Lemma 2.6]{Zim1}.

\begin{lemma}\label{cvxest1}
Let $D\subset\R^d$ $(d\geq3)$ be a hyperbolic convex domain, and $p,q\in D$. Let $L$ be the real line containing $p$ and $q$, and $\xi\in L\backslash D$ then
$$\frac{1}{2}\left|\log\left( \frac{||\xi-p||}{||\xi-q||}\right)\right|\leq\rho_D(p,q).$$
\end{lemma}
\proof
Up to a rigid change of coordinates, we can suppose that $\xi$ is the origin and $D\subset \H^d$.
Clearly we have
$$\rho_{\H^d}(x,y)\leq\rho_D(x,y).$$
Let $\gamma\colon [0,1]\to \H^d$ be a piecewise $\mathcal{C}^1$ curve joining $p=(p_1,p_2)$ and $q=(q_1,q_2)$. If we write $\gamma(t)=(\gamma_1(t),\gamma_2(t))\in\R\times\R^{d-1}$, then by Lemma \ref{halfp}
$$\int_0^1g_{\H^d}(\gamma(t),\dot\gamma(t))dt\geq\int_0^1\frac{|\dot\gamma_1(t)|}{2\gamma_1(t)}dt\geq\frac{1}{2}\left|\log\left(\frac{p_1}{q_1}\right)\right|,$$
so if we take the infimum over all curves $\gamma$, we obtain
$$\rho_{\H^d}(x,y)\geq\left|\log\left(\frac{p_1}{q_1}\right)\right|.$$
We conclude noticing that by Thales's theorem $$\frac{||p||}{||q||}=\frac{p_1}{q_1}.$$
\endproof

Now we will derive an upper bound estimate. Let $D\subset\R^d$ $(d\geq3)$ be a domain, $x\in D$ and $v\in\R^d$ nonzero. Let  $\Lambda\in G_2(\R^d)$, then
$$\delta_D(x,\Lambda):=\inf\{||x-y||: y\in (x+\Lambda)\backslash D\}$$
and
$$\hat{\delta}_D(x,v):=\max\{\delta_{D}(x,\Lambda):\Lambda\in G_2(\R^d)\ \mbox{s.t.}\ v\in\Lambda\}.$$

\begin{lemma}\label{cvxest2}
Let $D\subset\R^d$ $(d\geq3)$ be a domain. Then for all $x\in D, v\in\R^d$ nonzero we have
$$g_D(x,v)\leq \frac{||v||}{\hat{\delta}_D(x,v)}.$$
\end{lemma}
\proof First of all we may assume that $||v||=1$. Let $\Lambda$ be a 2-plane containing $v$ such that $\hat{\delta}_D(x,v)=\delta_D(x,\Lambda)=:r$. Let $u\in\Lambda$ be such that $||u||=1$ and $\langle u,v\rangle=0$. Consider the conformal harmonic disk $f\colon \D\to\R^d$ given by $$f(a,b)=x+r(a v+b u).$$
Notice that $f(\D)\subset D$ and $f_x(0)=rv$, so by definition
$$g(x,v)\leq \frac{1}{r}.$$

\endproof

The following lemma on the construction of quasi-geodesics is analogous to the lemma proved in the complex case in \cite[Lemma 3.2]{Zim1}.

\begin{lemma}[Quasi-geodesics]\label{qgcvx}
Let $D\subset\R^d$ $(d\geq3)$ be a convex domain, $p\in D$ and $\xi\in\partial D$. Let $\epsilon>0$ such that
$$\hat{\delta}_D(p,\xi-p)\geq\epsilon||\xi-p||.$$ Then the curve $\sigma\colon [0,+\infty)\to D$ given by 
$$\sigma(t)=\xi+e^{-2t}(p-\xi)$$
has the property
$$|t-s|\leq \rho_D(\sigma(s),\sigma(t))\leq 2\epsilon^{-1}|t-s|,$$
for all $t,s\geq0$. Therefore, $\sigma$ is $(2\epsilon^{-1},0)$ quasi-geodesic. 
\end{lemma}
\proof
The lower bound is a simple application of Lemma \ref{cvxest1}.

For the upper bound, let $\Lambda\in G_2(\R^d)$ be such that $$\delta_{D}(p,\Lambda)=\hat{\delta}_D(p,\xi-p)\geq\epsilon||\xi-p||.$$
Since $D$ is convex, we have for all $u\geq0$ $$\delta_D(\sigma(u),\Lambda)\geq\epsilon e^{-2u}||\xi-p||,$$
so by Lemma \ref{cvxest2}
$$g_D(\sigma(u),\dot{\sigma}(u))\leq \frac{||\dot{\sigma}(u)||}{\hat{\delta}_D(\sigma(u),\dot{\sigma}(u))}\leq \frac{||\dot{\sigma}(u)||}{\delta_D(\sigma(u),\Lambda)}\leq\frac{2e^{-2u}||\xi-p||}{\epsilon e^{-2u}||\xi-p||}=2\epsilon^{-1}.$$
Finally, for all $0\leq s\leq t$
$$\rho_D(\sigma(s),\sigma(t))\leq\int_s^t g_D(\sigma(u),\dot{\sigma}(u))du\leq 2\epsilon^{-1}(t-s).$$
\endproof
We can now prove Theorem \ref{MAIN2}. In the following proof, we denote the segment between $x,y\in\R^d$  by $[x,y]$.

\begin{proof}[Proof of Theorem \ref{MAIN2}]
The main idea of the proof, which closely resembles \cite[Theorem 3.1]{Zim1}, is to construct a sequence of quasi-geodesic triangles that are not uniformly slim, in contrast to Theorem \ref{thinqtriangle}.

First of all, by Proposition \ref{affinediskboundary}, $\partial D$ contains an affine disk in the boundary.	
Let $\Lambda$ be an affine plane such that $\partial D\cap \Lambda=:F$ has non-empty interior $F^{\mathrm{o}}$ in $\Lambda$. Fix $\xi\in F^{\mathrm{o}}$,$\eta\in F\backslash F^o$ and $p\in D$.

By Lemma \ref{qgcvx} the curves
$$\gamma(t)=\xi+e^{-2t}(p-\xi)$$
and
$$\sigma(t)=\eta+e^{-2t}(p-\eta)$$
are $(A,0)$ quasi geodesics for some $A>1$.\medskip

\textbf{Claim 1:} There exists $A'\geq A$ and $T_0>0$ such that for all $T\geq T_0$, the segment $[\gamma(T),\sigma(T)]$ can be parameterized to be a $(A',0) $ quasi geodesic.
\proof
First of all notice that the vector $\gamma(T)-\sigma(T)$ is parallel to $\Lambda$. Therefore, by the convexity of $D$, there exists $c>0$ and $T_0>0$ such that for all $T>T_0$ we have $$\hat{\delta}_D(\gamma(T),\gamma(T)-\sigma(T))\geq c>0.$$
This implies, by Lemma \ref{qgcvx}, that the segment  $[\gamma(T),\sigma(T)]$ can be parametrized to be a $(A',0)$ with $A'>A$.
\endproof

\textbf{Claim 2:} $\lim_{t\to+\infty}\rho_D(\gamma(t),\sigma[0,+\infty))=+\infty$.
\proof
For all $t\geq0$, let $s_t\geq0$ such that $$\rho_D(\gamma(t),\sigma[0,+\infty))=\rho_D(\gamma(t),\sigma(s_t)).$$
By contradiction, suppose that there exists $t_n\nearrow+\infty$ such that $\rho_D(\gamma(t_n),\sigma(s_{t_n}))<R$. Then
$$\rho_D(\sigma(s_{t_n}),p)\geq\rho_D(\gamma(t_n),p)-\rho_D(\gamma(t_n),\sigma(s_{t_n}))>\rho_D(\gamma({t_n}),p)-R\to+\infty,$$
so $s_{t_n}\to+\infty$. Consider the intersection between the line joining $\gamma({t_n})$ and $\sigma(s_{t_n})$ with $\partial D$, and denote by $\eta_{t_n}$ the point in such intersection that is closest to $\sigma(s_{t_n})$. Clearly $\eta_{t_n}\to\eta$ and $||\gamma({t_n})-\eta_{t_n}||\to||\xi-\eta||>0$, whereas $||\sigma({t_n})-\eta_{t_n}||\to0$. Using Lemma \ref{cvxest1},  we obtain
$$\rho_D(\gamma({t_n}),\sigma(s_{t_n}))\geq\frac{1}{2}\left|\log\left( \frac{||\gamma({t_n})-\eta_{t_n}||}{||\sigma({t_n})-\eta_{t_n}||}\right)\right|\to+\infty,$$
obtaining a contradiction.
\endproof

\textbf{Claim 3:} For any $M>0$ there exists $T>T_0$ such that the quasi-geodesic triangle with sides $[p,\gamma(T)], [\gamma(T),\sigma(T)]$ and $[\sigma(T),p]$ is not $M$-thin.

\proof
By Claim 2, there exist $t>0$ such that $\rho_D(\gamma(t),\sigma[0,+\infty))>M$. Now by Lemma \ref{cvxest1} we have
$$\lim_{T\to+\infty}\rho_D(\gamma(t),[\gamma(T),\sigma(T)])=+\infty,$$
so there exists $T>t$ such that 
$$\rho_D(\gamma(t),[\gamma(T),\sigma(T)]\cup[\sigma(T),p])>M,$$
which means that the quasi-geodesic triangle $[p,\gamma(T)], [\gamma(T),\sigma(T)], [\sigma(T),p]$ is not $M$-thin.
\endproof

To summarize, there exists $A'>1$ and a family of $(A',0)$ quasi-geodesic triangles that are not $M$-thin for all $M>0$, so by Theorem \ref{thinqtriangle} the metric space $(D,\rho_D)$ is not Gromov hyperbolic.
\end{proof}

\section{The Minimal metric and the Hilbert metric in convex domains}
Let $D \subset \mathbb{R}^d$ be a convex domain, and denote by $\overline{D}^*$ and $\partial^* D$ its closure and boundary, respectively, in the one-point compactification of $\mathbb{R}^d$. There exists a natural pseudodistance $H_D$ called the \textit{Hilbert metric} of $D$. Given two distinct points $x, y \in D$, let $L_{x,y}$ be the real line passing through $x$ and $y$, and let $a, b \in \partial^* D$ be the endpoints of $L_{x,y}$ in $\overline{D}^*$, ordered as $a, x, y, b$. Then the Hilbert metric is defined to be equal to $0$ if $x = y$ and, if $x \neq y$, it is defined by the formula

$$H_D(x,y):=\frac{1}{2}\log\left(\frac{||x-b||}{||x-a||}\frac{||y-a||}{||y-b||}\right)$$
with the convention that
$$\frac{||x-\infty||}{||y-\infty||}=\frac{||y-\infty||}{||x-\infty||}=1.$$

Notice that if $D$ does not contain any affine real lines, then $H_D(x,y) > 0$ for all distinct points $x, y \in D$.

Furthermore, $H_D$ is also the intrinsic distance associated with the following Finsler metric. Let $L_{x,v} := x + \mathbb{R} v$, where $a, b \in \partial^* D$ are the endpoints of $L_{x,v}$ in $\overline{D}^*$. We define
$$h_D(x,v):=\frac{||v||}{2}\left(\frac{1}{||x-a||}+\frac{1}{||x-b||}\right),$$
with the convention that $$\frac{1}{||x-\infty||}=0$$

Since convex domains naturally exhibit two metrics (the minimal metric and the Hilbert metric), it is natural to study their relationships.

In the unit Euclidean ball of $\R^d$ $(d\geq3)$ the Hilbert metric turns out to be the Beltrami-Cayley metric of $\B^d$, so by Example \ref{exball} in $\B^d$ the minimal metric and the Hilbert metric coincide. They also coincide in the half-space $\H^d$ (Lemma \ref{halfp}). It is unclear whether these are the only two cases where this occurs.

However, for a generic convex domain, we always have the following inequality.

\begin{proposition}\label{gvsh}
Let $D\subset\R^d$ $(d\geq3)$ be a convex domain. Then 
$$h_D(x,v)\leq2g_D(x,v), \ \ \ x\in D,v\in\R^d$$
and
$$H_D(x,y)\leq2\rho_D(x,y), \ \ \ x,y\in D.$$
\end{proposition}
\proof
Let $x,y\in D$ be two distinct points and $a,b\in\partial^*D$ as usual. By Lemma \ref{cvxest1} we have
$$\rho_D(x,y)\geq\frac{1}{2}\log\left(\frac{||y-a||}{||x-a||}\right), \ \ \ \rho_D(x,y)\geq\frac{1}{2}\log\left(\frac{||x-b||}{||y-b||}\right).$$
Adding these two inequalities, we obtain $2\rho_D(x,y)\geq H_D(x,y)$.

Consider the line $L:=x+\R v$ and $a,b\in\partial^* D$ as usual. If $a=b=\infty$ then $h_D(x,v)=0$, so there is nothing to prove. Suppose $a\in\R^d$; by a rigid change of coordinates, we can assume that $a$ is the origin and $D\subset \H^d$. By Lemma \ref{halfp} and the Thales's theorem, we have
$$g_D(x,v)\geq g_{\H^d}(x,v)=\frac{|v_1|}{2x_1}=\frac{||v||}{2||x-a||}.$$
Similarly, if $b\in\R^d$, we obtain 
$$g_D(x,v)\geq\frac{||v||}{2||x-b||}.$$
Adding these two inequalities, we obtain $2g_D(x,y)\geq h_D(x,y)$.
\endproof

It is not clear whether the constant 2 in the previous proposition is sharp, but Example 12.2 in \cite {DrFor} shows that the sharp constant must be greater than or equal to $\frac{4}{\pi}>1$. Furthermore, $\frac{4}{\pi}$ is sharp constant of the Schwarz lemma for bounded harmonic functions on the disk (see for example \cite{KV}). Therefore, it is natural to formulate the following conjecture.

\begin{conjecture}
Let $D\subset\R^d$ $(d\geq3)$ be a convex domain. Then 
$$h_D(x,v)\leq\frac{4}{\pi}g_D(x,v), \ \ x\in D,v\in\R^d$$
and
$$H_D(x,y)\leq\frac{4}{\pi}\rho_D(x,y), \ \ \ x,y\in D.$$
Moreover, the two inequalities are sharp.
\end{conjecture}

It is also natural to wonder when the two metrics are bilipschitz. However, this is not generally the case for a generic convex domain, as exemplified by the following example.

\begin{example}
Let $D\subset\R^d$ $(d\geq3)$ be a convex domain. Then the inequality
$$g_D(x,v)\leq A h_D(x,v),$$
for some $A\geq1$ is in general false. For example, let $D\subset\R^{d-1}$ be a bounded convex domain and consider $D'=\R\times D$. Clearly,
$$h_{D'}(p,e_1)=0$$
but $$g_{D'}(p,e_1)>0$$ since $D'$ is hyperbolic, as it does not contain any 2-dimensional affine subspaces (Theorem \ref{thDrFor}).
\end{example}

Conversely, for strongly convex domains, we have the following result.

\begin{proposition}
	Let $D\subset\R^d$ $(d\geq3)$ be a strongly convex domain. then there exists $A\geq1$ such that
	$$g_D(x,v)\leq Ah_D(x,v), \ \ x\in D,v\in\R^d.$$
\end{proposition}

The proposition follows directly from the Theorem \ref{BSMCestimate}, as strongly convex domains are strongly minimally convex, along with the following estimate.

\begin{lemma}
Let $D\subset\R^d$ $(d\geq2)$ be a strongly convex domain. Then there exists $\epsilon>0$ and $A\geq1$ such that
for all $x\in N_\epsilon(\partial D)\cap D$ and $v\in\R^d$, we have
$$h_D(x,v)\geq A^{-1}\left(\frac{||v_N||}{2\delta_D(x)}+\frac{||v_T||}{\delta_D(x)^{1/2}} \right).$$
\end{lemma}
\proof
We may assume $v$ nonzero.
Define $$\delta_D(x,v):=\inf\{||x-y||:y\in\partial D\cap(x+\R v)\}.$$
Clearly,
$$h_D(x,v)\geq\frac{||v||}{\delta_D(x,v)}.$$
By a simple geometric argument, for all $x\in N_\epsilon(\partial D)\cap D$ we have
$$\frac{\delta_D(x,v)}{\delta_D(x)}\leq \frac{||v||}{||v_N||}$$
so
$$\frac{||v||}{\delta_D(x,v)}\geq\frac{||v_N||}{\delta_D(x)}.$$
Since $D$ is strongly convex, we can find $c>1$ such that $\delta_D(x,v)\leq c\delta_D(x)^{1/2}$. Thus, we have
$$\frac{||v||}{\delta_D(x,v)}\geq c^{-1}\frac{||v_T||}{\delta_D(x)^{1/2}}$$
which implies
$$h_D(x,v)\geq\frac{||v||}{\delta_D(x,v)}\geq \frac{c^{-1}}{2}\left(\frac{||v_N||}{2\delta_D(x)}+\frac{||v_T||}{\delta_D(x)^{1/2}} \right).$$
\endproof

\end{document}